\documentclass[12pt,reqno]{amsart}
\usepackage{amsmath,amsfonts,amsthm,amsopn,amssymb,mathrsfs,enumerate,color}
\usepackage{amsmath}
\usepackage{comment}

\usepackage{cite,marginnote}
\usepackage{mathtools} 

\usepackage{color,graphicx,enumerate}
\usepackage[colorlinks=true,urlcolor=blue,
citecolor=red,linkcolor=blue,linktocpage,pdfpagelabels,
bookmarksnumbered,bookmarksopen]{hyperref}
\usepackage[english]{babel}

\usepackage[left=2.9cm,right=2.9cm,top=2.8cm,bottom=2.8cm]{geometry}

\numberwithin{equation}{section}

\makeindex
\newtheorem{theorem}{Theorem}[section]
\newtheorem{corollary}{Corollary}[section]
\newtheorem{remark}{Remark}[section]

\newtheorem{definition}{Definition}[section]
\newtheorem{proposition}{Proposition}[section]
\newtheorem{lemma}{Lemma}[section]
\newtheorem{iteration lemma}{iteration Lemma}[section]

\newcommand{\R}{\mathbb R}

\newcommand{\bt}{\begin{theorem}}
\newcommand{\et}{\end{theorem}}
\newcommand{\bl}{\begin{lemma}}
\newcommand{\el}{\end{lemma}}
\newcommand{\bd}{\begin{definition}}
\newcommand{\ed}{\end{definition}}
\newcommand{\bc}{\begin{corollary}}
\newcommand{\ec}{\end{corollary}}
\newcommand{\bp}{\begin{proof}}
\newcommand{\ep}{\end{proof}}
\newcommand{\bx}{\begin{example}}
\newcommand{\ex}{\end{example}}
\newcommand{\bi}{\begin{exercise}}
\newcommand{\ei}{\end{exercise}}
\newcommand{\bo}{\begin{proposition}}
\newcommand{\eo}{\end{proposition}}
\newcommand{\br}{\begin{remark}}
\newcommand{\er}{\end{remark}}
\newcommand{\beq}{\begin{equation}}
\newcommand{\eeq}{\end{equation}}
\newcommand{\ba}{\begin{align}}
\newcommand{\ea}{\end{align}}
\newcommand{\bn}{\begin{enumerate}}
\newcommand{\en}{\end{enumerate}}
\newcommand{\bg}{\begin{align*}}
\newcommand{\bcs}{\begin{cases}}
\newcommand{\ecs}{\end{cases}}

\newcommand{\bean}{\begin{eqnarray*}}
\newcommand{\eean}{\end{eqnarray*}}

\def\R{\mathbb{R}}

\newcommand{\HH}{\mathbb{H}}      

\def\bd{\mathrm{bd}\,}

\newcommand{\Sp}{\mathbb{S}}

\usepackage{amsmath,amssymb,amsthm,mathtools}
\usepackage{hyperref}

\usepackage{tikz}
\usetikzlibrary{arrows.meta,decorations.pathreplacing,positioning}

\theoremstyle{remark}

\title[Sobolev--Escobar bridge inequality]{Stability of the Sobolev--Escobar bridge inequality}
	
\author[S.~Fan]{Song Fan}
\author[G-D.~Li]{Gui-Dong Li}
\author[J.~J.~Zhang]{Jianjun Zhang}
	
\address[S.~Fan]{\newline\indent School  of  Mathematics  and  Statistics
\newline\indent
Guizhou University
\newline\indent
Guiyang, 550025, Guizhou, PR China}
\email{\href{mailto:doraemonsong77@gmail.com}{doraemonsong77@gmail.com}}

\address[G-D.~Li]{\newline\indent School  of  Mathematics  and  Statistics
\newline\indent
Guizhou University
\newline\indent
Guiyang, 550025, Guizhou, PR China}
\email{\href{mailto:bestdong123@163.com}{bestdong123@163.com}}

\address[J.~J.~Zhang]{\newline\indent College of Mathematics and Statistics
\newline\indent
Chongqing Jiaotong University
\newline\indent
Xuefu, Nan'an, 400074, Chongqing, PR China}
\email{\href{mailto:zhangjianjun09@tsinghua.org.cn}{zhangjianjun09@tsinghua.org.cn}}

\begin{document}

\begin{abstract}
We study the local stability of the bridge family
\[
\Phi(T):=\inf_{u\in\mathcal A_T}\|\nabla u\|_{L^2(\mathbb R^n_+)}, \qquad T>0,\quad n\ge3,
\]
where
\[
\mathcal A_T
:=
\Bigl\{
u\in \dot H^1(\mathbb R^n_+):
\|u\|_{L^{\frac{2n}{n-2}}(\mathbb{R}_{+}^n)}=1,\ \|u\|_{L^{\frac{2(n-1)}{n-2}}(\partial\mathbb{R}_{+}^n)}=T
\Bigr\},
\]
and \(\dot H^1(\mathbb R^n_+)\) is the completion of \(C_c^\infty(\overline{\mathbb R^n_+})\) in the norm \(\|\nabla \varphi\|_{L^2(\mathbb R^n_+)}\).
Let \(\mathcal M_T\) denote the set of minimizers of \(\Phi(T)\). We prove that, for every \(T\neq T_E\), there exists \(\alpha_T>0\) such that
\[
\|\nabla u\|_{L^2(\mathbb{R}_{+}^n)}^2-\Phi(T)^2
\ge \alpha_T\,d_T(u,\mathcal M_T)^2
+o\!\bigl(d_T(u,\mathcal M_T)^2\bigr)
\qquad\text{for all }u\in\mathcal A_T,
\]
where \(T_E\) is the Escobar threshold and \(d_T\) is the distance in \(\dot H^1(\mathbb R^n_+)\).

\vskip0.23in

\noindent{\it   {\bf Key  words:} Quantitative Stability, Sobolev inequality, Trace inequality.}   

\end{abstract}

\maketitle

\section{Introduction}
Sharp critical Sobolev inequalities play a central role in elliptic PDE,
calculus of variations, and conformal geometry. Besides the determination of
the optimal constants and the classification of extremals, a fundamental
question is quantitative stability: if a function nearly attains equality, must
it be close, modulo the natural symmetries, to the manifold of extremals?

In the present paper we study this question for a family of sharp inequalities
on the half-space
\[
\HH:=\{x=(x_1,x')\in \R\times \R^{n-1}:x_1>0\},
\qquad
\partial\HH:=\{x_1=0\},
\qquad n\ge 3,
\]
We write \(\dot H^1(\HH)\) for the Sobolev space endowed with
\(\|u\|_{\dot H^1(\HH)}:=\|\nabla u\|_{L^2(\HH)}\).
The two endpoint models are classical. The first is the sharp Sobolev
inequality on \(\mathbb R^n\),
\begin{equation}\label{eq:intro-Sobolev}
\|\nabla \varphi\|_{L^2(\R^n)}^2
\ge
S_n\,\|\varphi\|_{L^{2^*}(\R^n)}^2
\qquad\text{for all }\varphi\in \dot H^1(\R^n), \qquad 2^*=\frac{2n}{n-2}.
\end{equation}
The optimal constant and the extremals in \eqref{eq:intro-Sobolev} were identified by Aubin--Talenti \cite{Aubin76,Talenti76}. 
A quantitative stability result for \eqref{eq:intro-Sobolev} was later established by Bianchi--Egnell \cite{BE91}.
The second endpoint is the boundary-critical Sobolev trace inequality on
\(\HH\),
\begin{equation}\label{eq:intro-Escobar}
\|\nabla \varphi\|_{L^2(\HH)}^2
\ge
E_n\,\|\varphi\|_{L^{2^\#}(\partial\HH)}^2
\qquad\text{for all }\varphi\in \dot H^1(\HH), \qquad 2^\#=\frac{2(n-1)}{n-2}.
\end{equation}
The sharp form of this inequality was proved by Escobar \cite{Escobar88}.
Quantitative stability for the Escobar inequality \eqref{eq:intro-Escobar} and related trace inequalities has also 
been investigated recently; see for example \cite{Ho22,ZZZ25,BCB26}.

The Sobolev--Escobar bridge inequality couples these two critical quantities \eqref{eq:intro-Sobolev}--\eqref{eq:intro-Escobar}. 
In the Hilbertian case \(p=2\), it is given by the doubly constrained minimization problem
\begin{equation}\label{eq:intro-Phi}
\Phi(T)=
\inf_{u\in\mathcal A_T}\|\nabla u\|_{L^2(\HH)},
\qquad T>0,
\end{equation}
where
\begin{equation}\label{eq:intro-AT}
\mathcal A_T
=
\Bigl\{
u\in \dot H^1(\HH):
\|u\|_{L^{2^*}(\HH)}=1,
\ \|u\|_{L^{2^\#}(\partial\HH)}=T
\Bigr\}.
\end{equation}
The endpoint \(T=0\) corresponds to the sharp Sobolev inequality \eqref{eq:intro-Sobolev},  
\(T=T_E>0\) recovers the sharp Escobar trace inequality \eqref{eq:intro-Escobar}.
Thus \eqref{eq:intro-Phi} furnishes a one-parameter interpolation between the bulk-critical and boundary-critical problems.
The bridge problem \eqref{eq:intro-Phi} is already well understood at the level of existence and classification of extremals. In the general \(p\)-case,
this was established by Maggi--Neumayer \cite{MN17}, while in the conformally invariant case \(p=2\) the classification goes back to
Carlen--Loss \cite{CL94}. For further results on the bridge problem, see also
Maggi--Neumayer--Tomasetti \cite{MNT23}.

We are interested  in the quantitative stability of \eqref{eq:intro-Phi}. When \(T\neq T_E\), 
no Bianchi--Egnell type quantitative stability theorem \cite{BE91} seems to be available for the bridge problem \eqref{eq:intro-Phi},
 even in the Hilbert case \(p=2\). At the special value \(T=T_E\) and for \(p=2\), the bridge problem
\eqref{eq:intro-Phi} recovers the Escobar regime \eqref{eq:intro-Escobar},
and the corresponding stability was proved in \cite{Ho22}.
The present paper is concerned with the case of \(T\neq T_E\).

The main new difficulty is that, away from the endpoint \(T=T_E\), \eqref{eq:intro-Phi} no longer arises 
as a perturbation of a single scale-invariant inequality, but instead as a doubly constrained family of minimizers.
In particular, the second variation must be analyzed on the linearized constraint space \eqref{eq:XU-def},
 modulo the symmetry directions generated by tangential translations and critical dilations. 
The key issue, therefore, is to prove that the linearized operator is coercive transverse to the minimizing manifold.

Our approach is to reduce this coercivity problem to a Robin spectral problem on a model ball. 
After writing  \(\psi=U\hat\phi\), where \(U\) is a bridge minimizer, the ground-state transform converts the
Hessian into a quadratic form \eqref{eq:QU-metric} for the conformal metric \(g_U=U^{4/(n-2)}|dx|^2\). 
Because the \(p=2\) bridge extremals are explicitly classified, this conformal metric is isometric to a geodesic ball in \(\mathbb S^n\) when \(0<T<T_E\), 
and to a geodesic ball in \(\mathbb H^n\) when \(T>T_E\); see Subsection~\ref{subsec:model}. 
In these variables, the linearized equation becomes a Robin problem on a bounded model space; see Proposition~\ref{prop:model-reduction},
and the symmetry-generated Jacobi fields become explicit \(\ell=1\) modes; see Appendix \ref{sec:appendix}.
The proof is then reduced to a kernel identification and spectral-gap argument for the corresponding Robin quadratic form.

For \(\lambda>0\) and \(z\in\R^{n-1}\), define the natural action
\[
\bigl((\lambda,z)\cdot u\bigr)(x_1,x')
:=
\lambda^{\frac{n-2}{2}}
u\bigl(\lambda x_1,\lambda(x'-z)\bigr).
\]
This action preserves the Dirichlet norm, the \(L^{2^*}(\HH)\)-norm, and the \(L^{2^\#}(\partial\HH)\)-norm. 
Fix \(T>0\), and let \(U_T\in\mathcal A_T\) be a positive
minimizer. By the classification of the half-space bridge minimizers due to
Carlen--Loss \cite{CL94} in the conformal case \(p=2\), and Maggi--Neumayer
\cite{MN17} in the general case, the full minimizing set is
\[
\mathcal M_T
:=
\{\,\pm(\lambda,z)\cdot U_T:\ \lambda>0,\ z\in\R^{n-1}\,\}.
\]
For \(u\in\mathcal A_T\), we define the distance to the minimizing manifold by
\[
d_T(u)
:=
\inf_{v\in\mathcal M_T}\|\nabla(u-v)\|_{L^2(\HH)},
\]
and the bridge deficit by
\[
\delta_T(u)
:=
\|\nabla u\|_{L^2(\HH)}^2-\Phi(T)^2.
\]

Our main result is a local Bianchi--Egnell type estimate.

\begin{theorem}
\label{thm:local-stability}
Fix \(T>0\) with \(T\neq T_E\). Then there exists \(\alpha_T>0\) such that
\[
\delta_T(u)
\ge
\alpha_T\,d_T(u)^2
+
o\!\bigl(d_T(u)^2\bigr)
\qquad
\text{as } u\in\mathcal A_T,\ d_T(u)\to0 .
\]
\end{theorem}

\begin{remark}
In the first version of this paper, we left open the question whether there
exists a constant \(C_T>0\) such that
\begin{equation}
  \label{eq:stability-global}
  \delta_T(u)\ge C_T d_T(u)^2
\qquad\text{for all }u\in\mathcal A_T
\end{equation}
This question has since been answered affirmatively by Neumayer
\cite{Neumayer26}. More precisely, Neumayer proved a qualitative stability theorem
for the Sobolev--Escobar bridge inequality: if \(u_k\in\mathcal A_T\) and
\(\delta_T(u_k)\to0\), then \(d_T(u_k)\to0\). In the conformal case \(p=2\),
combining this qualitative compactness result with
Theorem~\ref{thm:local-stability} gives \eqref{eq:stability-global}
for some constant \(C_T>0\).
\end{remark}

\section{Reduction to a Robin problem}

\subsection{Euler--Lagrange equation and second variation}
In this section \(U\) denotes a positive bridge minimizer. The negative
component is obtained by the sign change \(U\mapsto -U\), and will not be
distinguished in the notation.

Let \(0<U\in\mathcal A_T\) be a minimizer of
\(E(u)=\int_{\HH}|\nabla u|^2\,dx\).
Set \(G(u)=\int_{\HH}|u|^{2^*}\,dx\), \(H(u)=\int_{\partial\HH}|u|^{2^\#}\,dS\).
Since \(U\) is a constrained minimizer, there exist \(\mu,\eta\in\R\) such that \(U\) is a critical point of
\begin{equation}\label{eq:L-equation}
 \mathcal L(u):=E(u)-\mu G(u)-\eta H(u).
\end{equation}
Moreover, \(U\) satisfies
\begin{equation}\label{eq:EL-bridge}
\left\{
\begin{aligned}
-\Delta U &= \lambda\,U^{2^*-1} &&\text{in }\HH,\\
\partial_\nu U &= \sigma\,U^{2^\#-1} &&\text{on }\partial\HH,
\end{aligned}
\right.
\end{equation}
where
\(\lambda=\frac{\mu\,2^*}{2}\), \(\sigma=\frac{\eta\,2^\#}{2}\).
The linearized constraint space at \(U\) is
\begin{equation}\label{eq:XU-def}
\mathcal X_U
:=
\left\{
\psi\in \dot H^1(\HH):
\int_{\HH}U^{2^*-1}\psi\,dx=0,\ 
\int_{\partial\HH}U^{2^\#-1}\psi\,dS=0
\right\}.
\end{equation}

The Hessian of the Lagrangian at \(U\) is the quadratic form
\begin{equation}\label{eq:QU-def}
Q_U(\psi)
=
\int_{\HH}|\nabla\psi|^2\,dx
-\lambda(2^*-1)\int_{\HH}U^{2^*-2}\psi^2\,dx
-\sigma(2^\#-1)\int_{\partial\HH}U^{2^\#-2}\psi^2\,dS .
\end{equation}

We next use the standard ground-state transform \(\psi=U\hat\phi\); see, for example,
\cite[Lemma~2.4]{PT06}. Since
\[
|\nabla(U\hat\phi)|^2
=
U^2|\nabla\hat\phi|^2+\nabla U\cdot \nabla(U\hat\phi^2),
\]
an integration by parts, together with \eqref{eq:EL-bridge}, yields
\begin{equation}\label{eq:ground-trans}
\int_{\HH}|\nabla\psi|^2\,dx
=
\int_{\HH}U^2|\nabla\hat\phi|^2\,dx
+\lambda\int_{\HH}U^{2^*}\hat\phi^2\,dx
+\sigma\int_{\partial\HH}U^{2^\#}\hat\phi^2\,dS .
\end{equation}
Substituting \eqref{eq:ground-trans} into \eqref{eq:QU-def} and using
\(\psi=U\hat\phi\), we obtain
\begin{equation}\label{eq:QU-conj}
Q_U(\psi)
=
\int_{\HH}U^2|\nabla\hat\phi|^2\,dx
-(2^*-2)\lambda\int_{\HH}U^{2^*}\hat\phi^2\,dx
-\beta\int_{\partial\HH}U^{2^\#}\hat\phi^2\,dS,
\quad
\beta=(2^\#-2)\sigma.
\end{equation}

\subsection{The conformal model}\label{subsec:model}

Assume from now on that \(T\neq T_E\), and define
\(g_U:=U^{\frac{4}{n-2}}|dx|^2\) on \(\HH\).
If \(\psi\in C_c^\infty(\overline{\HH})\) and \(\hat\phi=\psi/U\), then
\begin{equation*}
dV_{g_U}=U^{2^*}\,dx,
\qquad
dS_{g_U}=U^{2^\#}\,dS,
\qquad
\int_{\HH}U^2|\nabla\hat\phi|^2\,dx
=
\int_{\HH}|\nabla_{g_U}\hat\phi|^2\,dV_{g_U}.
\end{equation*}
Hence \eqref{eq:QU-conj} becomes
\begin{equation}\label{eq:QU-metric}
Q_U(\psi)
=
\int_{\HH}|\nabla_{g_U}\hat\phi|^2\,dV_{g_U}
-(2^*-2)\lambda\int_{\HH}\hat\phi^2\,dV_{g_U}
-\beta\int_{\partial\HH}\hat\phi^2\,dS_{g_U}.
\end{equation}

We now recall the explicit conformal models in the two nondegenerate branches.

\textbf{The spherical branch \(0<T<T_E\).}
After a tangential translation and a critical dilation,
\begin{equation}\label{eq:sphere}
U(x)=C_t^{\mathrm S}\,(1+|x-te_1|^2)^{-\frac{n-2}{2}}
\end{equation}
for some \(t\in\R\) and \(C_t^{\mathrm S}>0\), where \(t\) and \(C_t^{\mathrm S}\) are determined by the two
constraints \eqref{eq:intro-AT};  see Carlen--Loss \cite{CL94} and Maggi--Neumayer \cite{MN17}.  Writing \(y:=x-te_1\), one has
\[
g_U
=
\bigl(C_t^{\mathrm S}\bigr)^{\frac{4}{n-2}}(1+|y|^2)^{-2}|dy|^2.
\]
If \(\Pi_S^{-1}:\R^n\to \Sp^n\setminus\{N\}\subset\R^{n+1}\) denotes the inverse stereographic projection from the north pole \(N\), namely
\[
\Pi_S^{-1}(y)
=
\left(
\frac{2y}{1+|y|^2},
\frac{|y|^2-1}{1+|y|^2}
\right),
\qquad y\in\R^n,
\]
see for example \cite[Section~4.4]{LL01}, then
\[
(\Pi_S^{-1})^*g_{\Sp^n}
=
\frac{4}{(1+|y|^2)^2}|dy|^2.
\]
Thus,
\[
g_U=\alpha_t^{\mathrm S}(\Pi_S^{-1})^*g_{\Sp^n},
\qquad
\alpha_t^{\mathrm S}=\frac{(C_t^{\mathrm S})^{4/(n-2)}}{4}.
\]
Therefore \(\Pi_S^{-1}\) identifies \((\{y_1>-t\},g_U)\) with a geodesic ball in \((\Sp^n,\alpha_t^{\mathrm S}g_1)\). 
After composing with a suitable rotation \(O_t\in O(n+1)\), we obtain an isometry
\[
F_t^{\mathrm S}:(\HH,g_U)\longrightarrow (B_1(R_t),\,\alpha_t^{\mathrm S}g_1),
\]
where \(B_1(R_t)\subset \Sp^n\) denotes the geodesic ball of radius \(R_t\) centered at the north pole.

\textbf{The hyperbolic branch \(T>T_E\).}
After a tangential translation and a critical dilation,
\begin{equation}\label{eq:hyperbolic}
U(x)=C_t^{\mathrm H}\,(|x-te_1|^2-1)^{-\frac{n-2}{2}}
\end{equation}
for some \(t<-1\) and \(C_t^{\mathrm H}>0\), again determined by the two constraints \eqref{eq:intro-AT}; see
\cite{CL94,MN17}. Writing \(y=x-te_1\),
\[
g_U
=
\bigl(C_t^{\mathrm H}\bigr)^{\frac{4}{n-2}}(|y|^2-1)^{-2}|dy|^2.
\]
If \(\Pi_H^{-1}:\{y\in\R^n:\ |y|>1\}\to \mathbb H^n\subset\R^{n+1}\) denotes the inverse hyperbolic stereographic projection, namely
\[
\Pi_H^{-1}(y)
=
\left(
\frac{2y}{|y|^2-1},
\frac{|y|^2+1}{|y|^2-1}
\right),
\qquad |y|>1,
\]
see for example \cite[Section~8.8]{CFKP97}, then
\[
(\Pi_H^{-1})^*g_{\mathbb H^n}
=
\frac{4}{(|y|^2-1)^2}|dy|^2.
\]
Thus,
\[
g_U=\alpha_t^{\mathrm H}(\Pi_H^{-1})^*g_{\mathbb H^n},
\qquad
\alpha_t^{\mathrm H}=\frac{(C_t^{\mathrm H})^{4/(n-2)}}{4}.
\]
Therefore \(\Pi_H^{-1}\) identifies \((\{y_1>-t\},g_U)\) with a geodesic ball in \((\mathbb H^n,\alpha_t^{\mathrm H}g_{-1})\).
 After composing with a suitable Lorentz isometry \(L_t\), we obtain an isometry
\[
F_t^{\mathrm H}:(\HH,g_U)\longrightarrow (B_{-1}(R_t),\,\alpha_t^{\mathrm H}g_{-1}),
\]
where
\(
B_{-1}(R_t)\subset\mathbb H^n
\)
denotes the geodesic ball of radius \(R_t\) centered at a fixed base point \(o=e_{n+1}\).

In both cases we write
\[
\vartheta=
\begin{cases}
1,&0<T<T_E,\\
-1,&T>T_E,
\end{cases}
\qquad
\tilde g_t:=\alpha_t g_{\vartheta},
\qquad
\kappa_t:=\frac{\vartheta}{\alpha_t},
\]
and let
\begin{equation}\label{eq:reduction-map}
F:(\HH,g_U)\longrightarrow (B_{\vartheta}(R_t),\tilde g_t)
\end{equation}
denote the corresponding isometry. Since \(U\) solves \eqref{eq:EL-bridge}, the scalar-curvature
formula for the conformal metric \(g_U\) yields
\begin{equation}\label{eq:lambda-k}
\lambda=\frac{n(n-2)}{4}\,\kappa_t,
\qquad
(2^*-2)\lambda=n\kappa_t.
\end{equation}
Substituting \eqref{eq:lambda-k} into \eqref{eq:QU-metric}, we obtain
\begin{equation}\label{eq:QU-metric-k}
Q_U(\psi)
=
\int_{\HH}|\nabla_{g_U}\hat\phi|^2\,dV_{g_U}
-
n\kappa_t\int_{\HH}\hat\phi^2\,dV_{g_U}
-
\beta\int_{\partial\HH}\hat\phi^2\,dS_{g_U}.
\end{equation}

\subsection{Reduction to the Robin problem on the model ball}

\begin{proposition}\label{prop:model-reduction}
Assume \(T\neq T_E\). For \(\psi\in C_c^\infty(\overline{\HH})\), set
\(
\hat\phi=\frac{\psi}{U},
\phi=\hat\phi\circ F^{-1}.
\)
Then
\[
Q_U(\psi)=Q_t(\phi),
\]
where
\begin{equation}\label{eq:QU-ball}
Q_t(\phi)= 
\int_{B_{\vartheta}(R_t)}|\nabla_{\tilde g_t}\phi|^2\,dV_{\tilde g_t}
-
n\kappa_t\int_{B_{\vartheta}(R_t)}\phi^2\,dV_{\tilde g_t}
-
\beta\int_{\partial B_{\vartheta}(R_t)}\phi^2\,dS_{\tilde g_t}.
\end{equation}
Moreover, the associated bilinear form 
\(\mathcal B_t(\phi,\zeta)=0\) for all \(\zeta\in H^1(B_{\vartheta}(R_t))\) if and only if \(\phi\) is a weak
solution of
\begin{equation}\label{eq:Robin}
\left\{
\begin{aligned}
-\Delta_{\tilde g_t}\phi &= n\kappa_t\,\phi &&\text{in }B_{\vartheta}(R_t),\\
\partial_{\nu_{\tilde g_t}}\phi &= \beta\,\phi &&\text{on }\partial B_{\vartheta}(R_t).
\end{aligned}
\right.
\end{equation}
Moreover, the linearized constraints \eqref{eq:XU-def} are transformed into
\begin{equation}\label{eq:ball-constraints}
\int_{B_{\vartheta}(R_t)}\phi\,dV_{\tilde g_t}=0,
\qquad
\int_{\partial B_{\vartheta}(R_t)}\phi\,dS_{\tilde g_t}=0.
\end{equation}
\end{proposition}

\begin{proof}
The identity \(Q_U(\psi)=Q_t(\phi)\) follows immediately from
\eqref{eq:QU-metric-k} by transporting the three terms through the isometry \(F\). The weak
Robin \eqref{eq:Robin} is exactly the Euler--Lagrange equation associated with the bilinear form \(\mathcal B_t\).
Finally, since \(\psi=U\hat\phi\), the linearized constraints become
\begin{equation}\label{eq:const-F}
\int_{\HH}\hat\phi\,dV_{g_U}=0,
\qquad
\int_{\partial\HH}\hat\phi\,dS_{g_U}=0.
\end{equation}
Transporting \eqref{eq:const-F} by \(F\) gives \eqref{eq:ball-constraints}.
\end{proof}

\section{Proof of Theorem~\ref{thm:local-stability}}
For a fixed \(U\in\mathcal M_T\),  one has
\[
T_U\mathcal M_T=\operatorname{span}\{Z_0,Z_2,\dots,Z_n\},
\]
where
\begin{equation}\label{eq:tangent-all}
Z_0=\frac{n-2}{2}U+x\cdot\nabla U,
\qquad
Z_i=\partial_{x_i}U,\quad i=2,\dots,n.
\end{equation}

\begin{lemma}\label{lem:nearest}
There exists \(\varepsilon_T>0\) such that whenever
\(u\in\mathcal A_T\),
\(d_T(u)<\varepsilon_T\),
the infimum
\[
d_T(u)^2=\inf_{U\in\mathcal M_T}\|\nabla(u-U)\|_{L^2(\HH)}^2
\]
is attained at some \(U\in\mathcal M_T\). Moreover, if \(u=U+\psi\) with \(U\) a nearest point, then
\begin{equation}\label{eq:orth-D12}
\psi\perp_{\dot H^1} T_U\mathcal M_T,
\end{equation}
where \(T_U\mathcal M_T\) denotes the tangent space to the connected component of
\(\mathcal M_T\) containing \(U\).
\end{lemma}

\begin{proof}
Since
\(\mathcal M_T=\mathcal M_T^+\cup(-\mathcal M_T^+)\), where
\[
\mathcal M_T^+:=
\{(\lambda,z)\cdot U_T:\lambda>0,\ z\in\R^{n-1}\},
\]
and since \(u\mapsto -u\) preserves \(\mathcal A_T\) and \(d_T\), we may, after replacing
\(u\) by \(-u\) if necessary, assume that the positive component realizes the distance to
\(\mathcal M_T\). Let
\[
\Xi:(0,\infty)\times\R^{n-1}\to \dot H^1(\HH),
\qquad
\Xi(\lambda,z):=(\lambda,z)\cdot U_T,
\]
so that \(\mathcal M_T^+=\Xi((0,\infty)\times\R^{n-1})\). For fixed \(u\in\mathcal A_T\), define
\[
\mathscr F_u(\lambda,z):=\|\nabla(u-\Xi(\lambda,z))\|_{L^2(\HH)}^2.
\]
Then, in the present sign component,
\[
d_T(u)^2=\inf_{(\lambda,z)\in(0,\infty)\times\R^{n-1}}\mathscr F_u(\lambda,z).
\]
Choose \(U_*=\Xi(\lambda_*,z_*)\in\mathcal M_T^+\) such that
\(
\|\nabla(u-U_*)\|_{L^2(\HH)}\le 2d_T(u)
\).
Using the invariance of \(\mathcal A_T\), \(\mathcal M_T\), and \(d_T\) under tangential translations and
critical dilations, we may replace \(u\) by \((\lambda_*,z_*)^{-1}\cdot u\). Thus it is enough to treat the
case
\(
\|\nabla(u-U_T)\|_{L^2(\HH)}\le 2d_T(u)
\),
and hence
\begin{equation}\label{eq:F<4d2}
\mathscr F_u(1,0)\le 4d_T(u)^2.
\end{equation}

Let \((\lambda_k,z_k)\) be a minimizing sequence for \(\mathscr F_u\). We claim that, provided
\(d_T(u)\) is sufficiently small, the sequence \((\lambda_k,z_k)\) remains in a compact subset of
\((0,\infty)\times\R^{n-1}\).
Indeed, if \(|z_k|\to\infty\), then the translated profiles drift away tangentially, and therefore
\begin{equation}\label{eq:vanish}
\int_{\HH}\nabla u\cdot \nabla \Xi(\lambda_k,z_k)\,dx\to0. 
\end{equation}
Since \(\|\nabla\Xi(\lambda_k,z_k)\|_{L^2(\HH)}=\Phi(T)\), it follows that
\[
\mathscr F_u(\lambda_k,z_k)
=
\|\nabla u\|_{L^2(\HH)}^2+\Phi(T)^2
-2\int_{\HH}\nabla u\cdot \nabla \Xi(\lambda_k,z_k)\,dx
\to
\|\nabla u\|_{L^2(\HH)}^2+\Phi(T)^2.
\]
Since \(u\in\mathcal A_T\), one has \(\|\nabla u\|_{L^2(\HH)}\ge \Phi(T)\), and if \(d_T(u)\to0\) then
\(\|\nabla u\|_{L^2(\HH)}\to\Phi(T)\). Hence, for \(d_T(u)\) sufficiently small,
\[
\|\nabla u\|_{L^2(\HH)}^2+\Phi(T)^2 > 4d_T(u)^2,
\]
which contradicts \eqref{eq:F<4d2}. Therefore \((z_k)\) is bounded.
Similarly, if \(\lambda_k\to0\) or \(\lambda_k\to\infty\), then \eqref{eq:vanish} again holds.
The same contradiction shows that \((\lambda_k)\) stays in a compact subinterval of \((0,\infty)\).

Thus, after passing to a subsequence,
\((\lambda_k,z_k)\to(\lambda_0,z_0)\)
for some \((\lambda_0,z_0)\in(0,\infty)\times\R^{n-1}\). Since \(\Xi\) is continuous as a map into
\(\dot H^1(\HH)\), the function \(\mathscr F_u\) is continuous, and therefore the minimum is attained at
\(U:=\Xi(\lambda_0,z_0)\in\mathcal M_T\).

Finally, if \(u=U+\psi\) and \(U\) is a nearest point, then \(U\) is a critical point of the restriction of
\(W\mapsto \|u-W\|_{\dot H^1(\HH)}^2\)
to the connected component of \(\mathcal M_T\) containing \(U\). Hence its differential vanishes on
\(T_U\mathcal M_T\), which is exactly \eqref{eq:orth-D12}.
This proves the lemma.
\end{proof}

\begin{lemma}
\label{lem:quadratic}
Let \(T>0\), and let \(U\) be a positive bridge minimizer. Then there exist
\(\varepsilon_U,C_U>0\) such that, whenever
\[
\psi\in \dot H^1(\HH),\qquad
\psi\perp_{\dot H^1}T_U\mathcal M_T,\qquad
\|\nabla\psi\|_{L^2(\HH)}\le \varepsilon_U,\qquad
U+\psi\in\mathcal A_T,
\]
there exists
\(W\in \mathcal X_U\cap (T_U\mathcal M_T)^{\perp_{\dot H^1}}\)
such that
such that
\begin{equation}\label{eq:quadratic-est}
\|\nabla(\psi-W)\|_{L^2(\HH)}
\le
C_U\,\|\nabla\psi\|_{L^2(\HH)}^2.
\end{equation}
In particular,
\begin{equation}\label{eq:W-psi}
\|\nabla W\|_{L^2(\HH)}^2
=
\|\nabla\psi\|_{L^2(\HH)}^2
+
o\!\bigl(\|\nabla\psi\|_{L^2(\HH)}^2\bigr)
\qquad\text{as~ }\|\nabla\psi\|_{L^2(\HH)}\to0.
\end{equation}
\end{lemma}

\begin{proof}
Set
\(\mathcal C_U=(T_U\mathcal M_T)^{\perp_{\dot H^1}}\),
and define \(\mathfrak D:\dot H^1(\HH)\to\R^2\) by 
\[
\mathfrak D(\varphi)
:=
\left(
\int_{\HH}U^{2^*-1}\varphi\,dx,\ 
\int_{\partial\HH}U^{2^\#-1}\varphi\,dS
\right).
\]
Then \(\mathcal X_U=\ker\mathfrak D\). We claim that
\begin{equation}\label{eq:defect}
|\mathfrak D(\psi)|\le C\,\|\nabla\psi\|_{L^2(\HH)}^2.
\end{equation}
Since \(u=U+\psi\in\mathcal A_T\) and \(U\in\mathcal A_T\), one has
\(\|U+\psi\|_{L^{2^*}(\HH)}=\|U\|_{L^{2^*}(\HH)}\).
Set \(p=2^*\).  The pointwise expansion
\begin{equation}\label{eq:remainder-est}
\bigl||a+b|^p-a^p-pa^{p-1}b\bigr|
\le C\bigl(a^{p-2}|b|^2+|b|^p\bigr),
\qquad a\ge0,\quad b\in\mathbb R,\quad p>2 .
\end{equation}
yields
\[
\left|\,p\int_{\HH}U^{p-1}\psi\,dx\right|
\le
C\int_{\HH}U^{p-2}\psi^2\,dx
+
C\int_{\HH}|\psi|^p\,dx.
\]
By H\"older inequality and the  Sobolev inequality,
\[
\int_{\HH}U^{p-2}\psi^2\,dx
=O(\|\nabla\psi\|_{L^2(\HH)}^2),\quad \int_{\HH}|\psi|^p\,dx=
o(\|\nabla\psi\|_{L^2(\HH)}^2).
\]
Therefore
\begin{equation}\label{eq:bulk-remainder}
\int_{\HH}U^{2^*-1}\psi\,dx
=
O(\|\nabla\psi\|_{L^2(\HH)}^2).
\end{equation}
Similarly, by \eqref{eq:remainder-est} with \(q=2^\#\) and the trace inequality,  we obtain
\begin{equation}\label{eq:boundary-remainder}
\int_{\partial\HH}U^{2^\#-1}\psi\,dS
=
O(\|\nabla\psi\|_{L^2(\HH)}^2).
\end{equation}
Thus, \eqref{eq:bulk-remainder} and \eqref{eq:boundary-remainder} imply \eqref{eq:defect}.

We claim that the restriction \(\mathfrak D|_{\mathcal C_U}:\mathcal C_U\to\R^2\) is surjective. 
By duality, it is enough to show that if a linear functional
\(\ell\in(\R^2)^*\) vanishes on \(\mathfrak D(\mathcal C_U)\), then \(\ell=0\).
Write
\(
\ell(\xi_1,\xi_2)=\alpha \xi_1+\beta \xi_2
\)
for some \((\alpha,\beta)\in\R^2\), and assume that
\[
\alpha \int_{\HH}U^{2^*-1}\varphi\,dx
+
\beta \int_{\partial\HH}U^{2^\#-1}\varphi\,dS
=0
\qquad\text{for all } \varphi\in\mathcal C_U.
\]
Since \(T_U\mathcal M_T\subset \mathcal X_U=\ker\mathfrak D\), the same identity holds trivially for every
\(\varphi\in T_U\mathcal M_T\). Because
\(
\dot H^1(\HH)=\mathcal C_U\oplus T_U\mathcal M_T,
\)
it follows that
\[
\alpha \int_{\HH}U^{2^*-1}\varphi\,dx
+
\beta \int_{\partial\HH}U^{2^\#-1}\varphi\,dS
=0
\qquad\text{for all } \varphi\in\dot H^1(\HH).
\]
Now take \(\varphi\in C_c^\infty(\HH)\). Then the boundary term vanishes, and we obtain
\[
\alpha \int_{\HH}U^{2^*-1}\varphi\,dx=0
\qquad\text{for all } \varphi\in C_c^\infty(\HH).
\]
Since \(U>0\) in \(\HH\), this implies \(\alpha=0\). Therefore
\[
\beta \int_{\partial\HH}U^{2^\#-1}\varphi\,dS=0
\qquad\text{for all } \varphi\in\dot H^1(\HH).
\]
Choosing \(\varphi\) with nontrivial trace on \(\partial\HH\), we conclude that \(\beta=0\). Hence
\(\ell=0\), and therefore \(\mathfrak D|_{\mathcal C_U}\) is surjective.

Since \(\mathfrak D|_{\mathcal C_U}\) is surjective and \(\R^2\) is finite-dimensional, there exists a bounded linear right inverse 
\(\mathcal R_U:\R^2\to\mathcal C_U\). Let \(\tau=\mathcal R_U(\mathfrak D(\psi))\) and \(W=\psi-\tau\). 
Then \(W\in\mathcal C_U\) and \(\mathfrak D(W)=0\), so \(W\in \mathcal X_U\cap (T_U\mathcal M_T)^{\perp_{\dot H^1}}\). 
Furthermore, by \eqref{eq:defect},
\[
\|\nabla(\psi-W)\|_{L^2(\HH)}
=
\|\nabla\tau\|_{L^2(\HH)}
\le
C_U\,|\mathfrak D(\psi)|
\le
C_U\,\|\nabla\psi\|_{L^2(\HH)}^2.
\]
This proves \eqref{eq:quadratic-est}, while \eqref{eq:W-psi} follows immediately.
\end{proof}

\begin{proposition}
\label{prop:spectral-gap}
Assume \(T\neq T_E\). Then there exists \(\Lambda_T>0\) such that, for every
positive minimizer \(U\in\mathcal M_T\) and every
\(W\in \mathcal X_U\cap (T_U\mathcal M_T)^{\perp_{\dot H^1}}\),
one has
\begin{equation}\label{eq:spectral-gap}
Q_U(W)\ge \Lambda_T\,\|\nabla W\|_{L^2(\HH)}^2.
\end{equation}
\end{proposition}

\begin{proof}
By symmetry invariance, it is enough to work at one fixed positive minimizer
\(U\in\mathcal M_T\).
Set 
\[
\mathcal H_{0,t}
:=
\left\{
\phi\in H^1(B_{\vartheta}(R_t)):
\eqref{eq:ball-constraints}~ \text{holds}
\right\}, \quad \mathcal T_U(W):=\Bigl(\frac{W}{U}\Bigr)\circ F^{-1}.
\]
By Proposition~\ref{prop:model-reduction}, \(\mathcal T_U\) identifies \(\mathcal X_U\) 
with \(\mathcal H_{0,t}\), and \(Q_U\) with \(\mathcal Q_t\). By Proposition~\ref{prop:kernel-identify},
\begin{equation}\label{eq:kernel-phi}
\ker(\mathcal Q_t|\mathcal H_{0,t})
=
\operatorname{span}\{\phi_0,\phi_2,\dots,\phi_n\}.
\end{equation}
We now set
\[
\mathcal Z_{t,U}
=
\left\{
\phi\in \mathcal H_{0,t}:
\Lambda_j(\phi)=0,\ j\in\{0,2,\dots,n\}
\right\},
\]
where
\(\Lambda_j(\phi):=\bigl\langle U(\phi\circ F),\,Z_j\bigr\rangle_{\dot H^1(\HH)}\).
Since each \(\Lambda_j\) is a continuous linear functional on \(H^1(B_{\vartheta}(R_t))\),
the space \(\mathcal Z_{t,U}\) is weakly closed. It is therefore enough to prove that 
there exists \(\mu_T>0\) such that
\begin{equation}\label{eq:coercive-z}
\mathcal Q_t(\phi)\ge \mu_T\|\nabla_{\tilde g_t}\phi\|_{L^2(B_{\vartheta}(R_t))}^2
\qquad
\qquad\text{for all } \phi\in \mathcal Z_{t,U}.
\end{equation}
Assume by contradiction that \eqref{eq:coercive-z} fails. Then there exists
\((\varphi_m)\subset \mathcal Z_{t,U}\) such that
\begin{equation}\label{eq:contradict-setting}
\|\nabla_{\tilde g_t}\varphi_m\|_{L^2(B_{\vartheta}(R_t))}=1,
\qquad
\mathcal Q_t(\varphi_m)\to 0.
\end{equation}
Since \(\varphi_m\in \mathcal H_{0,t}\), the Poincar\'e inequality on the bounded connected
Riemannian domain \((B_{\vartheta}(R_t),\tilde g_t)\), see \cite[Theorem~2.10]{Hebey99}, gives
\[
\|\varphi_m\|_{L^2(B_{\vartheta}(R_t),dV_{\tilde g_t})}
\le C_T\|\nabla_{\tilde g_t}\varphi_m\|_{L^2(B_{\vartheta}(R_t),dV_{\tilde g_t})}
= C_T.
\]
Thus \((\varphi_m)\) is bounded in \(H^1(B_{\vartheta}(R_t))\). Passing to a subsequence,
\[
\varphi_m\rightharpoonup \varphi_\infty
\quad\text{weakly in }H^1(B_{\vartheta}(R_t)),
\]
\begin{equation}\label{eq:compact-strong}
\varphi_m\to \varphi_\infty
\quad\text{strongly in }L^2(B_{\vartheta}(R_t)) \cap L^2(\partial B_{\vartheta}(R_t)).
\end{equation}
Since \(\mathcal Z_{t,U}\) is weakly closed in \(\mathcal H_{0,t}\), \(\varphi_\infty\in \mathcal Z_{t,U}\).
By weak lower semicontinuity of the gradient term and \eqref{eq:compact-strong},
\(\mathcal Q_t(\varphi_\infty)\le \liminf_{m\to\infty}\mathcal Q_t(\varphi_m)=0\).
Because \(U\) is a constrained minimizer, one has \(\mathcal Q_t\ge 0\) on \(\mathcal H_{0,t}\). 
Hence \(\mathcal Q_t(\varphi_\infty)=0\),
so \(\varphi_\infty\in \ker(\mathcal Q_t|\mathcal H_{0,t})\). 

By \eqref{eq:kernel-phi},
\(\varphi_\infty=\sum_{j\in\{0,2,\dots,n\}} a_j\phi_j\). Set \(W_\infty:=U(\varphi_\infty\circ F)\).
Since \(\phi_j=\bigl(\frac{Z_j}{U}\bigr)\circ F^{-1}\), we have
\[
W_\infty=\sum_{j\in\{0,2,\dots,n\}} a_j Z_j \in T_U\mathcal M_T.
\]
But \(\varphi_\infty\in \mathcal Z_{t,U}\) implies that
\(W_\infty\in (T_U\mathcal M_T)^{\perp_{\dot H^1}}\).
Therefore \(W_\infty=0\), and so \(\varphi_\infty=0\).
Returning to \eqref{eq:QU-ball}, together with \eqref{eq:compact-strong} and 
\(\|\nabla_{\tilde g_t}\varphi_m\|_{L^2(B_{\vartheta}(R_t))}=1\),
gives \(\mathcal Q_t(\varphi_m)\to 1\), contradicting \eqref{eq:contradict-setting}. 
This proves \eqref{eq:coercive-z}.

Finally, let
\(W\in \mathcal X_U\cap (T_U\mathcal M_T)^{\perp_{\dot H^1}}\),
\(\phi:=\mathcal T_U(W)\in \mathcal Z_{t,U}\).
Then by \eqref{eq:coercive-z},
\[
Q_U(W)=\mathcal Q_t(\phi)\ge \mu_T\|\nabla_{\tilde g_t}\phi\|_{L^2(B_{\vartheta}(R_t))}^2.
\]
Moreover, by \eqref{eq:ground-trans} and \eqref{eq:reduction-map},
\begin{equation}\label{eq:W-ground}
\|\nabla W\|_{L^2(\HH)}^2
=
\int_{B_{\vartheta}(R_t)}|\nabla_{\tilde g_t}\phi|^2\,dV_{\tilde g_t}
+\lambda\int_{B_{\vartheta}(R_t)}\phi^2\,dV_{\tilde g_t}
+\sigma\int_{\partial B_{\vartheta}(R_t)}\phi^2\,dS_{\tilde g_t}.
\end{equation}
Since \(\phi\in\mathcal H_{0,t}\), Poincar\'e and the trace inequality yield
\begin{equation}\label{eq:pincare-trace}
\int_{B_{\vartheta}(R_t)}\phi^2\,dV_{\tilde g_t}
+\int_{\partial B_{\vartheta}(R_t)}\phi^2\,dS_{\tilde g_t}
\le C_T \int_{B_{\vartheta}(R_t)}|\nabla_{\tilde g_t}\phi|^2\,dV_{\tilde g_t}.
\end{equation}
Together \eqref{eq:W-ground} and \eqref{eq:pincare-trace}, with \(\lambda,\sigma\) fixed by \(T\), prove
\eqref{eq:spectral-gap}. This completes the proof.
\end{proof}

\subsection{Proof of Theorem~\ref{thm:local-stability}}
\begin{proof}
  By the sign reduction discussed above, we may assume that the nearest point
\(U\) is positive. Let \(u\in\mathcal A_T\) with \(d_T(u)\) sufficiently small,
and choose such a nearest point \(U\in\mathcal M_T\) as in
Lemma~\ref{lem:nearest}. Writing \(u=U+\psi\), we have
\[
\|\nabla\psi\|_{L^2(\HH)}=d_T(u),
\qquad
\psi\perp_{\dot H^1}T_U\mathcal M_T.
\]
Since \(u,U\in\mathcal A_T\), the constraint terms in the Lagrangian cancel. 
Hence Taylor expansion at \(U\), together with
\eqref{eq:L-equation} and \eqref{eq:QU-def}, gives
\begin{equation}\label{eq:Taylor-delta}
\delta_T(u)=Q_U(\psi)+o(\|\nabla\psi\|_{L^2(\HH)}^2).
\end{equation}
By Lemma~\ref{lem:quadratic}, there exists
\(W\in \mathcal X_U\cap (T_U\mathcal M_T)^{\perp_{\dot H^1}}\) such that
\[
\|\nabla(\psi-W)\|_{L^2(\HH)}=O(\|\nabla\psi\|_{L^2(\HH)}^2),
\qquad
\|\nabla W\|_{L^2(\HH)}^2
=
\|\nabla\psi\|_{L^2(\HH)}^2
+
o(\|\nabla\psi\|_{L^2(\HH)}^2).
\]
If \(\mathcal B_U\) is the bilinear form associated with \(Q_U\), then
\[
Q_U(\psi)-Q_U(W)=\mathcal B_U(\psi-W,\psi+W)
=o(\|\nabla\psi\|_{L^2(\HH)}^2).
\]
Therefore, by Proposition~\ref{prop:spectral-gap},
\[
\delta_T(u)
=
Q_U(W)+o(\|\nabla\psi\|_{L^2(\HH)}^2)
\ge
\Lambda_T\|\nabla W\|_{L^2(\HH)}^2
+
o(\|\nabla\psi\|_{L^2(\HH)}^2).
\]
Using \eqref{eq:W-psi} and \(\|\nabla\psi\|_{L^2(\HH)}=d_T(u)\), we obtain
\[
\delta_T(u)
\ge
\Lambda_T d_T(u)^2+o(d_T(u)^2).
\]
This proves the theorem.
\end{proof}

\appendix

\section{Kernel identification on the reduced side}\label{sec:appendix}
Throughout this appendix, \(U\in\mathcal A_T\) is a positive \(p=2\) bridge
minimizer with \(T\neq T_E\). By Proposition~\ref{prop:model-reduction}, the
kernel problem on \(\mathcal X_U\) is reduced to the Robin kernel problem on a
model ball. For simplicity we write this model as
\((B_\kappa(R),g_\kappa)\), where \(\kappa=\kappa(U)\neq0\) is the constant
sectional curvature.
More precisely, under the change of variables
\(\phi=\Bigl(\frac{\psi}{U}\Bigr)\circ F^{-1}\),
where \(F\) is given by \eqref{eq:reduction-map},
the quadratic form \(Q_U\) is transformed into the reduced quadratic form
\begin{equation}\label{eq:def-Qphi}
\mathcal Q(\phi)
=
\int_{B_\kappa(R)}|\nabla_{g_\kappa}\phi|^2\,dV_{g_\kappa}
-
n\kappa\int_{B_\kappa(R)}\phi^2\,dV_{g_\kappa}
-
\beta\int_{\partial B_\kappa(R)}\phi^2\,dS_{g_\kappa},
\end{equation}
and the linearized constraints become
\[
\mathcal H_0
:=
\left\{
\phi\in H^1(B_\kappa(R)):
\int_{B_\kappa(R)}\phi\,dV_{g_\kappa}=0,\ 
\int_{\partial B_\kappa(R)}\phi\,dS_{g_\kappa}=0
\right\}.
\]
Thus it is enough to identify \(\ker(\mathcal Q|\mathcal H_0)\).

We write
\[
\phi_j:=\left(\frac{Z_j}{U}\right)\circ F^{-1},
\qquad j\in\{0,2,\dots,n\},
\]
where  \(Z_j\) are defined in \eqref{eq:tangent-all}.
Since these fields  \(Z_j\) arise from tangential translations and critical dilations, they belong to
\(\ker(Q_U|\mathcal X_U)\), and hence
\begin{equation}\label{eq:phi-in-ker}
\phi_0,\phi_2,\dots,\phi_n\in \ker(\mathcal Q|\mathcal H_0).
\end{equation}
The reduced kernel equation is
\begin{equation}\label{eq:Robin-kernel}
\left\{
\begin{aligned}
-\Delta_{g_\kappa}\phi-\kappa n\,\phi&=0 &&\text{in }B_\kappa(R),\\
\partial_\nu\phi-\beta\phi&=0 &&\text{on }\partial B_\kappa(R),
\end{aligned}
\right.
\qquad
\phi\in\mathcal H_0 .
\end{equation}

In geodesic polar coordinates, one has, see for example \cite[Example~1.4.6]{Petersen16},
\begin{equation}\label{eq:geodesic-corr}
g_\kappa=dr^2+s_\kappa(r)^2g_{\Sp^{n-1}},
\end{equation}
where 
\begin{equation*}
 s_\kappa(r):=
\begin{cases}
\dfrac{1}{\sqrt{\kappa}}\sin(\sqrt{\kappa}\,r),&\kappa>0,\\
r,&\kappa=0,\\
\dfrac{1}{\sqrt{-\kappa}}\sinh(\sqrt{-\kappa}\,r),&\kappa<0.
\end{cases}
\end{equation*}

We analyze the reduced Robin problem \eqref{eq:Robin-kernel} by separation of variables in the geodesic polar coordinates \eqref{eq:geodesic-corr}. 
Expanding \(\phi\) in spherical harmonics on \(\Sp^{n-1}\), see for example \cite[Theorem~2.38 and Proposition~3.5]{AH12},
\[
\phi(r,\theta)=\sum_{\ell=0}^\infty\sum_{m=1}^{d_\ell}f_{\ell,m}(r)Y_{\ell,m}(\theta),
\qquad
-\Delta_{\Sp^{n-1}}Y_{\ell,m}=\ell(\ell+n-2)Y_{\ell,m},
\]
where \(d_\ell\) denotes the dimension of the space of spherical harmonics of degree \(\ell\) on \(\Sp^{n-1}\), namely
\begin{equation}\label{eq:dimension-sph}
d_\ell=\frac{(2\ell+n-2)(\ell+n-3)!}{\ell!(n-2)!}.
\end{equation}
By orthogonality, \eqref{eq:def-Qphi} decomposes as
\(\mathcal Q(\phi)=\sum_{\ell,m}\mathcal Q_\ell(f_{\ell,m})\),
where
\[
\mathcal Q_\ell(f)
=
\int_0^R
\left(
|f'|^2+\frac{\ell(\ell+n-2)}{s_\kappa(r)^2}f^2-\kappa n\,f^2
\right)
s_\kappa(r)^{n-1}\,dr
-\beta s_\kappa(R)^{n-1}f(R)^2.
\]
Moreover,
\begin{equation}\label{eq:increasing-Q}
\mathcal Q_{\ell+1}(f)-\mathcal Q_\ell(f)
=
(2\ell+n-1)\int_0^R s_\kappa(r)^{n-3}f(r)^2\,dr>0
\end{equation}
for every nonzero \(f\).

\begin{lemma}\label{lem:symmetry-l1}
Each \(\phi_j\) lies in the \(\ell=1\) angular
sector. Moreover, \(\phi_0,\phi_2,\dots,\phi_n\) are linearly independent.
\end{lemma}
\begin{proof}
We treat the two nondegenerate branches separately.

\medskip
\noindent
\textbf{Spherical branch \(0<T<T_E\).}
Writing \(y=x-te_1\), \eqref{eq:sphere} gives
\[
\frac{Z_0}{U}
=
\frac{n-2}{2}\,\frac{1-|y|^2-2ty_1}{1+|y|^2},
\qquad
\frac{Z_i}{U}=-(n-2)\frac{y_i}{1+|y|^2},
\qquad i=2,\dots,n.
\]
Let \(X_1',\dots,X_n'\) denote the first \(n\) ambient coordinate functions on \(\Sp^n\) after the
rotation \(O_t\), that is,
\[
X_\alpha'(\xi):=(O_t^{-1}\xi)_\alpha,
\qquad \alpha=1,\dots,n.
\]
By the definition of the spherical reduction map \(F=O_t\circ \Pi_S^{-1}(\,\cdot\,-te_1)\),
one has
\[
X_1'\circ F
=
-\frac{1-|y|^2-2ty_1}{\sqrt{1+t^2}\,(1+|y|^2)},
\qquad
X_i'\circ F=\frac{2y_i}{1+|y|^2},
\qquad i=2,\dots,n.
\]
Hence
\[
\phi_0=-\frac{n-2}{2}\sqrt{1+t^2}\,X_1',
\qquad
\phi_i=-\frac{n-2}{2}\,X_i',
\qquad i=2,\dots,n.
\]

Since \(F(\HH)\) is a centered geodesic ball, in the geodesic polar coordinates
of \(g_\kappa\) the first ambient coordinate functions are constant multiples of
\(s_\kappa(r)\theta_\alpha\), \(\alpha=1,\dots,n\).
Thus each \(\phi_j\) is of the form \(c\,s_\kappa(r)\theta_\alpha\), and hence belongs to the
\(\ell=1\) sector.

\medskip
\noindent
\textbf{Hyperbolic branch \(T>T_E\).}
Again writing \(y=x-te_1\), \eqref{eq:hyperbolic} yields
\[
\frac{Z_0}{U}
=
-\frac{n-2}{2}\,\frac{|y|^2+1+2ty_1}{|y|^2-1},
\qquad
\frac{Z_i}{U}=-(n-2)\frac{y_i}{|y|^2-1},
\qquad i=2,\dots,n.
\]
Let \(Y_1',\dots,Y_n'\) denote the first \(n\) ambient coordinate functions on \(\mathbb H^n\) after
the Lorentz isometry \(L_t\), that is,
\(Y_\alpha'(\xi):=(L_t^{-1}\xi)_\alpha\), \(\alpha=1,\dots,n\).

By the definition of the hyperbolic reduction map
\(F=L_t\circ \Pi_H^{-1}(\,\cdot\,-te_1)\),
one has
\[
Y_1'\circ F
=
\frac{|y|^2+1+2ty_1}{\sqrt{t^2-1}\,(|y|^2-1)},
\qquad
Y_i'\circ F=\frac{2y_i}{|y|^2-1},
\qquad i=2,\dots,n.
\]
Hence
\[
\phi_0=-\frac{n-2}{2}\sqrt{t^2-1}\,Y_1',
\qquad
\phi_i=-\frac{n-2}{2}\,Y_i',
\qquad i=2,\dots,n.
\]
Since \(F(\HH)\) is a centered geodesic ball, in the geodesic polar coordinates
of \(g_\kappa\) the first ambient coordinate functions are constant multiples of
\(s_\kappa(r)\theta_\alpha\), \(\alpha=1,\dots,n\).
Thus each \(\phi_j\) is of the form \(c\,s_\kappa(r)\theta_\alpha\), and hence belongs to the
\(\ell=1\) sector.

In both branches, \(\phi_0,\phi_2,\dots,\phi_n\) are nonzero multiples of
\(\theta_1,\theta_2,\dots,\theta_n\) times the same radial factor
\(s_\kappa(r)\), which is not identically zero. Their linear independence
therefore follows from the linear independence of the first spherical harmonics.
\end{proof}

\begin{lemma}\label{lem:sector-dim}
Fix \(\ell\ge1\). In the \(\ell\)-sector, the space of \(H^1\)-solutions of the reduced kernel equation \eqref{eq:Robin-kernel} is
either trivial or \(d_\ell\)-dimensional. 
\end{lemma}

\begin{proof}
Let
\[
\phi(r,\theta)=f(r)Y(\theta),
\qquad
-\Delta_{\Sp^{n-1}}Y=\ell(\ell+n-2)Y .
\]
Then by \eqref{eq:Robin-kernel}, \(f\) satisfies
\begin{equation}\label{eq:sector-ode-appendix}
f''(r)
+
(n-1)\frac{s_\kappa'(r)}{s_\kappa(r)}f'(r)
-
\frac{\ell(\ell+n-2)}{s_\kappa(r)^2}f(r)
+
\kappa n\,f(r)=0
\qquad (0<r<R).
\end{equation}
Since
\[
s_\kappa(r)=r+O(r^3),
\qquad
\frac{s_\kappa'(r)}{s_\kappa(r)}=\frac1r+O(r),
\qquad
\frac1{s_\kappa(r)^2}=\frac1{r^2}+O(1)
\quad\text{as }r\downarrow0,
\]
equation \eqref{eq:sector-ode-appendix} has a regular singular point at \(r=0\), and can be written in the form
\[
f''+\frac{n-1}{r}f'-\frac{\ell(\ell+n-2)}{r^2}f+a(r)f'+b(r)f=0,
\]
with \(a,b\) continuous near \(0\). By the classical Frobenius theory for regular singular equations, the
indicial equation is
\[
\alpha(\alpha-1)+(n-1)\alpha-\ell(\ell+n-2)=0,
\]
whose roots are
\(\alpha_+=\ell\), \(\alpha_-=-(\ell+n-2)\).
Hence there is a basis of local solutions of the form
\[
f_{\rm reg}(r)=r^\ell(1+o(1)),
\qquad
f_{\rm sing}(r)=r^{-(\ell+n-2)}(1+o(1))
\qquad (r\downarrow0).
\]

We claim that the singular branch is not \( H^1\)-admissible. Indeed,
\(f'_{\rm sing}(r)\sim r^{-(\ell+n-1)}\),
and since \(dV_{g_\kappa}\sim r^{n-1}\,dr\,d\theta\) near \(r=0\), we get
\[
\int_0^\varepsilon |f'_{\rm sing}(r)|^2\,r^{n-1}\,dr
\sim
\int_0^\varepsilon r^{-2(\ell+n-1)}r^{n-1}\,dr
=
\int_0^\varepsilon r^{-2\ell-n+1}\,dr
=
\infty
\]
for every \(\ell\ge0\) and \(n\ge3\). Thus only the regular branch can belong to \( H^1\).

It follows that the local \( H^1\)-solution space near \(r=0\) is one-dimensional. Hence any two global
\( H^1\)-solutions of \eqref{eq:sector-ode-appendix} are proportional, since on every interval
\([r_0,R]\subset(0,R]\) the equation \eqref{eq:sector-ode-appendix} is a regular second-order linear ODE and uniqueness for the Cauchy
problem applies. Therefore the admissible radial profile is unique up to a multiplicative constant.

Consequently, if the \(\ell\)-sector is nontrivial, its full \(H^1\)-solution space is
\[
\{\,f_\ell(r)Y(\theta): Y\in\mathcal Y_\ell\,\},\qquad \text{where~~}  \mathcal{Y}_{\ell}=\left\{Y:-\Delta_{\Sp^{n-1}} Y=\ell(\ell+n-2) Y\right\},
\]
and therefore has dimension \(d_\ell\).
\end{proof}

For the radial sector \(\ell=0\), no nontrivial element of \(\ker(\mathcal Q|\mathcal H_0)\) exists. Indeed, if
\[
\phi(x)=f(r)\in \ker(\mathcal Q|\mathcal H_0)
\]
is radial, then \(\phi\) is constant on \(\partial B_\kappa(R)\). Since \(\phi\in\mathcal H_0\), \(f(R)=0\). 
The Robin boundary condition \eqref{eq:Robin-kernel} gives
\(f'(R)=\beta f(R)=0\).
By uniqueness for the Cauchy problem for the radial ODE, it follows that \(f\equiv0\). Thus
\begin{equation}\label{eq:non-ell=0}
\ker(\mathcal Q|\mathcal H_0)\cap\{\ell=0\}=\{0\}.
\end{equation}

\begin{proposition}\label{prop:kernel-identify}
Assume \(T\neq T_E\). Then
\begin{equation}\label{eq:kernel-appendix}
\ker(\mathcal Q|\mathcal H_0)
=
\operatorname{span}\{\phi_0,\phi_2,\dots,\phi_n\}.
\end{equation}
\end{proposition}

\begin{proof}
By \eqref{eq:phi-in-ker} and Lemma~\ref{lem:symmetry-l1}, \(\phi_j\), \(j\in\{0,2,\dots,n\}\),
are linearly independent elements of \(\ker(\mathcal Q|\mathcal H_0)\), and each lies in the
\(\ell=1\) sector. In particular, the \(\ell=1\) kernel is nontrivial, so its lowest Rayleigh
level is \(0\).

Since \(\mathcal Q\ge0\) on \(\mathcal H_0\), the \(\ell=1\) sector has lowest
level \(0\). By \eqref{eq:increasing-Q}, for every \(\ell\ge2\) and every
nonzero radial profile \(f\), one has \(\mathcal Q_\ell(f)>\mathcal Q_1(f)\ge0\).
Hence no sector \(\ell\ge2\) can contain a kernel element. By
\eqref{eq:non-ell=0}, the radial sector also contributes no kernel element.
Therefore
\[
\ker(\mathcal Q|\mathcal H_0)\subset \{\ell=1\}.
\]

By Lemma~\ref{lem:sector-dim}, in the \(\ell=1\) sector the space of \(H^1\)-solutions of the
reduced kernel equation \eqref{eq:Robin-kernel} is either trivial or \(d_1\)-dimensional. Since this sector already
contains the nonzero kernel elements \(\phi_0,\phi_2,\dots,\phi_n\), it is nontrivial; hence
its dimension is exactly
\(d_1=\dim\mathcal Y_1=n\) by \eqref{eq:dimension-sph}.

Because \(\phi_0,\phi_2,\dots,\phi_n\) are \(n\) linearly independent kernel elements in the
\(\ell=1\) sector, they form a basis of the entire kernel. This is exactly \eqref{eq:kernel-appendix}.
This completes the proof.
\end{proof}

\noindent\textbf{Conflict of interest:}
Authors state no conflict of interest.
\\

\noindent\textbf{Data Availability Statement:}
Data sharing is not applicable to this article as no datasets were generated or analysed during the current study.

\noindent{\bf Acknowledgement.} 
The authors thank Robin Neumayer for helpful correspondence concerning the
global stability problem. G-D.~Li was supported by NSFC (No.12561019).

\providecommand{\href}[2]{#2}
\providecommand{\arxiv}[1]{\href{http://arxiv.org/abs/#1}{arXiv:#1}}
\providecommand{\url}[1]{\texttt{#1}}
\providecommand{\urlprefix}{DOI }

\end{document}